\documentclass[a4paper,12pt,reqno]{amsart}
\usepackage{a4wide}
\usepackage{amsmath}
\usepackage{amssymb}
\usepackage{amsthm}
\usepackage{latexsym}
\usepackage{graphicx}
\usepackage[english]{babel}

\newtheorem{thm}[subsection]{Theorem}
\newtheorem{lem}[subsection]{Lemma}
\newtheorem{cor} [subsection]{Corollary}
\theoremstyle{definition}
\newtheorem{dfn} [subsection]{Definition}
\theoremstyle{remark}
\newtheorem{obs} [subsection]{Remark}

\def\supp{\operatorname{supp}}

\def\reg{\operatorname{reg}}

\def\k {\mathrm{k}}

\DeclareMathOperator{\pd}{pd}
\DeclareMathOperator{\Tor}{Tor}

\numberwithin{equation}{section}

\begin{document}

\title[Graded Betti numbers of powers of path ideals of paths]{Graded Betti numbers of powers of path ideals of paths}
\author[Silviu B\u al\u anescu, Mircea Cimpoea\c s, and Thanh Vu]{Silviu B\u al\u anescu$^1$, Mircea Cimpoea\c s$^2$ and Thanh Vu$^3$}  
\date{}



\footnotetext[1]{ \emph{Silviu B\u al\u anescu}, National University of Science and Technology Politehnica Bucharest, Faculty of
Applied Sciences, 
Bucharest, 060042, E-mail: silviu.balanescu@stud.fsa.upb.ro}
\footnotetext[2]{ \emph{Mircea Cimpoea\c s}, National University of Science and Technology Politehnica Bucharest, Faculty of
Applied Sciences, 
Bucharest, 060042, Romania and Simion Stoilow Institute of Mathematics, Research unit 5, P.O.Box 1-764,
Bucharest 014700, Romania, E-mail: mircea.cimpoeas@upb.ro,\;mircea.cimpoeas@imar.ro}
\footnotetext[3]{ \emph{Thanh Vu}, Institute of Mathematics, VAST, 18 Hoang Quoc Viet, Hanoi, Vietnam 122300, E-mail: vuqthanh@gmail.com}
\dedicatory{Dedicated to the memory of professor J\"urgen Herzog (1941--2024)}
\begin{abstract}  Let $I_{n,m} = (x_1\cdots x_{m},x_2 \cdots x_{m+1},\ldots,x_{n+1}x_{n+2}\cdots x_{n+m})$ be the $m$-path ideal of a path of length $n + m-1$ over a polynomial ring $S = \k[x_1,\ldots,x_{n+m}]$. We compute all the graded Betti numbers of all powers of $I_{n,m}$.
\end{abstract}

\maketitle

\section{Introduction}

Let $S = \k[x_1,\ldots,x_n]$ be a standard graded polynomial ring over a field $\k$ and $M$ a finitely generated graded $S$-module. The Hilbert Syzygy theorem states that $M$ has a finite graded free resolution
$$0 \longleftarrow M \longleftarrow F_0 \longleftarrow \cdots \longleftarrow F_r \longleftarrow 0,$$
where $F_i = \oplus_j S(-j)^{\beta_{i,j}(M)}$. The numbers $\beta_{i,j}(M)$ are called the graded Betti numbers of $M$. From the resolution, Hilbert deduced that the Hilbert polynomial of $M$ is 
$$p_M(z) = \sum_{i=0}^r (-1)^i \sum_j \beta_{i,j}(M) \binom{z + n-1 -j}{n-1}.$$
Hence, from the graded Betti numbers of $M$, one can deduce the dimension, projective dimension, regularity, and multiplicity of $M$. 

In \cite{BC1, SL, SWL}, the authors compute the projective dimension, the regularity, and the multiplicity of powers of path ideals of paths, respectively. In this paper, we compute all the graded Betti numbers of powers of path ideals of paths, giving a unified approach to the mentioned work. We first introduce some notation. Let $m \ge 2$ and $n \ge 0$ be integers. We denote by 
$$I_{n,m} = (x_1\cdots x_{m},x_2 \cdots x_{m+1},\ldots,x_{n+1}x_{n+2}\cdots x_{n+m}) \subseteq S = \k[x_1,\ldots,x_{n+m}]$$
the $m$-path ideal of a path on $n+m$-vertices. By convention, for integers $a,b$, if $a < 0$, $b< 0$ or $a < b$ then $\binom{a}{b} = 0$. Our main result is:

\begin{thm}\label{teo1} Let $m \ge 2$, $n \ge 0$, $t \ge 1$ be integers. The Betti number $\beta_{i,j}(I_{n,m}^t)$ is $0$ unless $j = i + tm + (m-1) \ell$ for some integer $\ell$ such that $0 \le \ell \le i$ and 
$$\beta_{i,i+tm + (m-1)\ell}(I_{n,m}^t)=\binom{t+\ell-1}{\ell}\binom{n-\ell m}{i-\ell}\binom{n+t-\ell m-i+\ell}{t-i+2\ell}.$$    
\end{thm}

Our result also covers the result of Alilooee and Faridi \cite{AF} for the case of path ideals of paths. As in \cite{BCV}, where we compute all the graded Betti numbers of $n-1$ and $n-2$ path ideals of $n$-cycles, we employ the following procedure.

\begin{enumerate}
    \item We use the following general idea from \cite{NV}. To study the ideal $I^n$, we break $I = J + K$ in a nice way, then we study the family $J^s I^t$ as this family has the natural decomposition $J^s I^t = J^{s+1} I^{t-1} + J^s K^t$. We prove that these decompositions are Betti splittings. 
    \item We then compute the intersection $J^{s+1} I^{t-1} \cap J^s K^t$ and exhibit further Betti splitting for the intersection.
    \item To establish the Betti splittings, we use the $\Tor$-vanishing criteria in \cite{NV}.
\end{enumerate}
The main difficulties in the procedure are finding the intersection $J^{s+1} I^{t-1} \cap J^s K^t$ and finding further Betti splittings for the intersection. The family of path ideals of paths have a natural decomposition $I_{n,m} = I_{n-1,m} + (f_{n+1})$ where $f_{n+1} = x_{n+1} \cdots x_{n+m}$. We then follow the steps outlined above to find relevant Betti splittings. From that, we deduce recursive equations for the graded Betti numbers. We then deduce the generating function for the graded Betti numbers and arrive at the proof of Theorem \ref{teo1}. Consequently, we deduce the results of \cite{BC1, SL}. 

We structure the paper as follows. In section \ref{sec_pre}, we provide background and recall the results on Betti splittings. In section \ref{sec_betti}, we establish various Betti splittings for the powers of path ideals of paths. We then deduce Theorem \ref{teo1}. Finally, we derive the formula for the projective dimension and regularity of $I_{n,m}^t$. 

\section{Preliminaries}\label{sec_pre}

Throughout this section, we let $S = \k[x_1,\ldots, x_n]$ be the polynomial ring over an arbitrary field $\k$, with the standard grading.

\subsection*{Projective dimension and regularity} Let $M$ be a finitely generated graded $S$-module. For integers $i,j$ with $i \ge 0$, the $(i,j)$-graded Betti number of $M$ is defined by 
$$\beta_{i,j}(M) = \dim_\k \Tor_i^S (\k,M)_j.$$

The projective dimension of $M$, denoted by $\pd(M)$, and the Castfelnuovo-Mumford regularity of $M$, denoted by $\reg(M)$, are defined as follows:
\begin{align*}
    \pd_S (M) &= \sup \{i :\; \beta_{i,j}(M) \neq 0 \text{ for some } j\},\\
    \reg_S(M) &= \sup \{ j -i :\; \beta_{i,j} (M) \neq 0\}.
\end{align*}

The following result is well-known.

\begin{lem}\label{lem_mul_x} Let $x$ be a variable and $I$ a nonzero homogeneous ideal of $S$. Then 
$$\beta_{i,j}(xI)=\beta_{i,j-1}(I) \text{ for all } i\ge 0.$$
\end{lem}

\subsection*{Betti splittings} 

Betti splittings of monomial ideals were first introduced by Francisco, Ha, and Van Tuyl in \cite{FHV}, motivated by the work of Eliahou and Kervaire \cite{ek}. Betti splittings have been proven useful in studying the path ideals of graphs and their powers \cite{BCV, HV}. We recall the definition and the following results about Betti splittings, following \cite{NV}:

\begin{dfn}
Let $P, I, J$ be proper nonzero homogeneous ideals of $S$ with $P = I + J$. The decomposition $P = I +J$ is called a \emph{Betti splitting} if for all $i \ge 0$ we have $$\beta_i(P) = \beta_i(I) + \beta_i(J) + \beta_{i-1}(I \cap J).$$
\end{dfn}

\begin{dfn} Let $\varphi: M \to N$ be a morphism of finitely generated graded $S$-modules. We say that $\varphi$ is $\Tor$-vanishing if for all $i \ge 0$, we have $\Tor_{i}^S(\k,\varphi) = 0$.    
\end{dfn}
We have the following criterion of Nguyen and Vu \cite[Lemma 3.5]{NV}.
\begin{lem}\label{lem_splitting_criterion} Let $I,J$ be nonzero homogeneous ideals of $S$ and $P = I+J$. The decomposition $P = I +J$ is a Betti splitting if and only if the inclusion maps $I \cap J \to I$ and $I\cap J \to J$ are $\Tor$-vanishing.    
\end{lem}

\begin{obs}
    By \cite[Proposition 2.1]{FHV}, once we have $P = I + J$ is a Betti splitting, then the mapping cone construction for the map $I \cap J \to I \oplus J$ yields a minimal free resolution of $P$. In particular, we have 
    $$\beta_{i,j}(P) = \beta_{i,j}(I) + \beta_{i,j}(J) + \beta_{i-1,j}(I \cap J)$$
    for all non-negative integers $i,j$.
\end{obs}

If $u\in S$ is a monomial, the \emph{support} of $u$, denoted by $\supp (u)$ is the set of variables $x_i$ such that $x_i | u$. Also, if $J\subset S$ is a monomial ideal with the minimal generating set $G(J)=\{u_1,\ldots,u_m\}$, 
the \emph{support} of $J$ is $\supp(J)=\bigcup_{i=1}^m \supp(u_i)$. The following results are well-known, see e.g. \cite{HH}.

\begin{lem}\label{lem_colon_sum}  Let $I,J,L$ be monomial ideals and $f$ a monomial of $S$. Then 
\begin{enumerate}
    \item $(I+J) : f = (I : f) + (J:f)$.
    \item $(I + J) \cap L = (I \cap L) + (J \cap L)$.
    \item If $\supp (I) \cap \supp (J) = \emptyset$ then $I\cap J = IJ$.
    \item If $f = xg$ where $x$ is a variable and $x \notin \supp (I)$ then $I:f = I:g$.
\end{enumerate}
\end{lem}

\section{Graded Betti numbers of powers of $m$-path ideals of $n+m$-paths}\label{sec_betti}

In this section, we compute all the graded Betti numbers of powers of $m$-path ideals of $n+m$-paths. We first introduce some notations. Let $n \ge 0$, $m \ge 2$ be integers and $S = \k[x_1,\ldots,x_{n+m}]$. We denote by $f_1 = x_1 \cdots x_{m}, \ldots, f_{n+1} = x_{n+1} \cdots x_{n+m}$. Then the $m$-path ideal of the $n+m$-path is 
$$I_{n,m} = (f_1,\ldots,f_{n+1}).$$
First, we have 
\begin{lem}\label{lem_colon_1} Let $t > s \ge 0$ be natural numbers. Then $I_{n,m}^t : f_{n+1}^s = I_{n,m}^{t-s}$.    
\end{lem}
\begin{proof}
    The proof is similar to that of \cite[Lemma 2.1]{BC1}. For completeness, we give an argument here. Since $f_{n+1} \in I_{n,m}$, the left-hand side contains the right-hand side. Let $u$ be a monomial in $I_{n,m}^t:f_{n+1}^s$. We prove by induction on $s$ that $u \in I_{n,m}^{t-s}$. The base case $s = 0$ is obvious. Now assume that the statement holds for $s-1$. Since $u \in I_{n,m}^{t} : f_{n+1}^s$, there exists a minimal generator $w$ of $I_{n,m}^t$ such that $w | u f_{n+1}^s$. Let $k = \max \{j \mid x_j \in \supp (w)\}$. Since $w$ is a minimal generator of $I_{n,m}^t$, it follows that $f_{k-m+1} | w$, $w = f_{k-m+1} w'$ where $w'$ is a minimal generator of $I_{n,m}^{t-1}$, and $\supp (w') \subseteq \supp (w)$. In particular, we have 
    $$w' | f_{n+1}^{s-1} \cdot \frac{f_{n+1}}{\gcd(f_{n+1},f_{k-m+1})} u.$$
    Since $\supp(w') \cap \supp \left ( \frac{f_{n+1}}{\gcd(f_{n+1},f_{k-m+1})} \right ) = \emptyset$, we deduce that $w' | f_{n+1}^{s-1} u$. By induction, the conclusion follows.
\end{proof}

For simplicity of notations, in the next four lemmas, we set
$I = (f_1,\ldots,f_{n+1})$, $J = (f_1, \ldots, f_n)$, $A = (f_1, \ldots, f_{n-m})$ and $B = (g_{n-m+1}, \ldots,g_n)$ where $g_i = f_i/x_n$. In particular, we have 
\begin{align*}
    I &= J + (f_{n+1})\\
    J & = A + x_n B
\end{align*}

\begin{lem}\label{lem_colon_2} Assume that $t \ge s \ge 1$ be natural numbers. Then 
$$J^t : f_{n+1}^s = A^{t-s+1} (A + (x_n))^{s-1} + x_n^{s} J^{t-s}.$$
\end{lem}
\begin{proof} Since $f_{n+1} = x_{n+m} g$ where $g = g_n = x_{n+1} \cdots x_{n+m-1}$ and $x_{n+m} \notin \supp (J)$, by Lemma \ref{lem_colon_sum}, we deduce that $J^t : f_{n+1}^s = J^t : g^s$. Now, we have
$$J^t = A^t + A^{t-1} (x_nB) + \cdots + (x_nB)^t.$$
Since $\supp (g) \cap \supp (A) = \emptyset$, $g$ is a minimal generator of $B$, and $B$ is isomorphic to the $m-1$-path ideal of the $2m-2$-path, by Lemma \ref{lem_colon_sum} and Lemma \ref{lem_colon_1}, we deduce that 
$$A^i (x_n B)^{t-i} : g^s = \begin{cases}
    A^i x_n^{t-i} B^{t-i-s} & \text{ if } i \le t-s,\\
    A^i x_n^{t-i} & \text { if } i > t - s.    
\end{cases}$$
By Lemma \ref{lem_colon_sum}, we have 
\begin{align*}
    J^t : g^s &= \sum_{i=0}^t (A^i (x_nB)^{t-i} : g^s)\\
    &= \sum_{i = 0}^{t-s} A^i x_n^{t-i} B^{t-i - s} + \sum_{i=t-s+1}^{t} A^i x_n^{t-i} \\
    &= x_n^{s} (A+ x_nB)^{t-s} + A^{t-s+1} (A + (x_n))^{s-1}.
\end{align*}
The conclusion follows.
\end{proof}

\begin{lem}\label{lem_decomp1} Assume that $t \ge 1$ and $s \ge 0$ are integers. Then 
    $$J^s f_{n+1}^t \cap J^{s+1} I^{t-1} = f_{n+1}^t J^s ( A + (x_n))$$
    and the decomposition $J^s I^t = J^s f_{n+1}^t + J^{s+1} I^{t-1}$ is a Betti splitting. 
\end{lem}
\begin{proof}    We have $I^t = f_{n+1}^t + J I^{t-1}$. Hence, $J^s I^t = J^{s} f_{n+1}^t + J^{s+1} I^{t-1}$. We now prove the equality for the intersection. 

First, we have $f_{n+1} x_n \in J$ and $A \subseteq J$, hence 
$f_{n+1} (A + (x_n)) \subseteq J$. In particular, the left-hand side contains the right-hand side. 

It remains to show that 
\begin{equation}\label{eq_3_1}
    J^s \cap (J^{s+1} I^{t-1} : f_{n+1}^t) \subseteq J^s ( A + (x_n)).
\end{equation}
By Lemma \ref{lem_colon_sum}, we have
\begin{equation}\label{eq_3_2}
J^{s+1} I^{t-1} : f_{n+1}^t = \sum_{j=0}^{t-1} (J^{s+1 + j} f_{n+1}^{t-1-j}) : f_{n+1}^t =\sum_{j=0}^{t-1} (J^{s+1+j} : f_{n+1}^{j+1}).    
\end{equation}
By Lemma \ref{lem_colon_sum} and Lemma \ref{lem_colon_2}, we have 
\begin{align*}
    J^s \cap (J^{s+j+1} : f_{n+1}^{j+1}) & =  J^s \cap (x_n^{j+1} J^{s} + A^{s+1}(A+x_n)^{j}) \\
    &\subseteq x_n J^s  + A^{s+1} \subseteq J^s(A + (x_n) ).
\end{align*}
Together with Eq. \eqref{eq_3_2} and Lemma \ref{lem_colon_sum}, Eq. \eqref{eq_3_1} follows.

It remains to prove that the decomposition $J^s I^t = J^s f^t + J^{s+1} I^{t-1}$ is a Betti splitting. First, we have any non-zero syzygy of $f_{n+1}^t J^s(A + (x_n))$ has $x_{n+m}$-degree $t$, while any non-zero syzygy of $J^{s+1} I^{t-1}$ has $x_{n+m}$-degree at most $t- 1$. Hence, the inclusion map $f_{n+1}^t J^s(A + (x_n)) \to J^{s+1} I^{t-1}$ is $\Tor$-vanishing. By Lemma \ref{lem_decomp2} and Lemma \ref{lem_splitting_criterion}, it suffices to prove that the inclusion maps $J^s x_n \to J^s$ and $J^s A \to J^s$ are $\Tor$-vanishing. The first one is clear, the second one follows from \cite[Theorem 4.5]{NV} and the fact that the map $J^s A \to J^s$ factors through $J^s A \to J^{s+1} \to J^s$.
\end{proof}

\begin{lem}\label{lem_decomp2} We have $J^s(A + (x_n)) = A^{s+1} + x_n J^s$, $A^{s+1} \cap x_n J^s = x_n A^{s+1}$, and the decomposition $J^s(A + (x_n) = A^{s+1} + x_nJ^s$ is a Betti splitting.    
\end{lem}
\begin{proof}
   First, we prove that $J^s(A + (x_n) ) = A^{s+1} + x_nJ^s$. It suffices to prove that $J^s A \subseteq A^{s+1} + x_nJ^s$. Indeed, we have $J = A + x_nB$. Hence, $J^s = \sum_{i=0}^s A^{s-i} (x_nB)^{i}$. Now if $i = 0$, then we have the term $A^{s+1}$. Thus, we may assume that $i > 0$. Then, we have 
   $$A A^{s-i} (x_nB)^i = x_n (A^{s-i+1} (x_nB)^{i-1} B \subseteq x_n (A^{s-i+1} (x_nB)^{i-1} \subseteq x_n J^s.$$

   Now, we prove that $x_n J^s \cap A^{s+1} = x_n A^{s+1}$. Indeed, we have 
   $$x_n A^{s+1} \subseteq x_n J^s \cap A^{s+1} \subseteq x_n \cap A^{s+1} = x_n A^{s+1}.$$
   
   Since the map $A^{s+1} \to J^s$ factors through $A^{s+1} \to J^{s+1} \to J^s$, it is $\Tor$-vanishing by \cite[Theorem 4.5]{NV}. Hence, the decomposition $J^s(A + (x_n)) = A^{s+1} + x_n J^s$ is a Betti splitting by Lemma \ref{lem_splitting_criterion}. The conclusion follows.
\end{proof}

We now deduce recursive equations for the Betti numbers of powers of $I_{n,m}$. Let $a$ be a function on the tuple $(n,s,t,i,j)$ defined as follows. If either $n,s,t,i,j < 0$ then $a(n,s,t,i,j) = 0$. If $n,s,t,i,j \ge 0$ then 
\begin{align*}
    a(n,s,t,i,j) &= \beta_{i,j}(I_{n-1,m}^s I_{n,m}^t).
\end{align*}
Furthermore, we set 
$$    \tilde a(n,s,t,i,j)  = a(n,s,t,i,j) + a(n,s,t,i-1,j-1).$$
With these notations, we have
\begin{lem}\label{lem_rec_1} Assume that $t \ge 1$ and $n,s,i,j \ge 0$ are integers. We have 
    \begin{align*}
    a(n,s,t,i,j) & = a(n,s+1,t-1,i,j) + \tilde a(n-1,0,s,i,j-tm) \\
                & + \tilde a(n-m-1,0,s+1,i-1,j-tm).
\end{align*}
\end{lem}
\begin{proof} Note that $I = I_{n,m}$, $J = I_{n-1,m}$, and $A = I_{n-m-1,m}$. By Lemma \ref{lem_decomp1}, the decomposition $J^s I^t = J^s f_{n+1}^t + J^{s+1} I^{t-1}$ is a Betti splitting. Hence, we have 
    \begin{equation}
        a(n,s,t,i,j) = a(n,s+1,t-1,i,j) +  \beta_{i,j}(J^s f_{n+1}^t) + \beta_{i-1,j}(P)
    \end{equation}
    where $P = f_{n+1}^tJ^s(A + (x_n) )$. By definition and Lemma \ref{lem_mul_x}, $\beta_{i,j}(J^sf_{n+1}^t) = \beta_{i,j-tm}(J^s) = a(n-1,0,s,i,j-tm)$ and $\beta_{i-1,j} (P) = \beta_{i,j-tm}(Q)$ where $Q = J^s(A + (x_n) )$. By Lemma \ref{lem_decomp2}, the decomposition $Q = A^{s+1} + x_n J^{s}$ is a Betti splitting. Hence, 
    $$\beta_{i-1,j-tm}(Q) = \beta_{i-1,j-tm} (x_n J^s) + \beta_{i-1,j-tm}(A^{s+1}) + \beta_{i-2,j-tm}(x_nA^{s+1}).$$
    The conclusion follows from Lemma \ref{lem_mul_x}.
\end{proof}

We now denote by 
$$b(n,t,i,j) = a(n,0,t,i,j) = \beta_{i,j}(I_{n,m}^t).$$

\begin{lem}\label{lem_rec_2} For all $n,t,i,j \ge 0$ we have 
\begin{align*}
    b(n,t+1,i,j+m) & = b(n,t,i,j) + b(n-1,t+1,i,j+m) + b(n-1,t,i-1,j-1) \\
    & + b(n-m-1,t+1,i-1,j) + b(n-m-1,t+1,i-2,j-1).
\end{align*}    
\end{lem}
\begin{proof} Applying Lemma \ref{lem_rec_1}, we deduce that   
\begin{align*}
    a(n,0,t,i,j) &= a(n,t,0,i,j) + \sum_{s=0}^{t-1}\tilde a(n-1,0,s,i,j-(t-s)m) \\
     &+\sum_{s=1}^t\tilde a(n-m-1,0,s,i-1,j-(t-s+1)m).
\end{align*}
We also let $\tilde b(n,t,i,j) = b(n,t,i,j) + b(n,t,i-1,j-1)$. Then we have 
\begin{align*}
    b(n,t,i,j) = b(n-1,t,i,j) + &\sum_{s=0}^{t-1} \left [ \tilde b(n-1,s,i,j-(t-s)m) \right . \\
    &  + \left .\tilde b(n-m-1,s+1,i-1,j-(t-s)m) \right ].
\end{align*}
    
Getting the equation for $b(n,t+1,i,j+m)$ then subtracting that to the equation for $b(n,t,i,j)$ we deduce that 
\begin{align*}
     b(n,t+1,i,j+m) &=b(n-1,t+1,i,j+m) + b(n,t,i,j) - b(n-1,t,i,j) \\
    &+ \tilde b(n-1,t,i,j) + \tilde b(n-m-1,t+1,i-1,j).
\end{align*}
The conclusion follows.
\end{proof}

Since $I_{n,m}^t$ is generated in degree $tm$, for all $i \ge 0$, we have $\beta_{i,j} (I_m(P_{n+m}) ^t)  = 0 $ if $j < i + tm$. For convenient in deriving these Betti numbers, we let $c(n,t,i,j)  = b(n,t,i,i + tm + j)$ and $\Psi(x,y,z,w)$ be the generating function
\begin{align*}
    \Psi (x,y,z,w) &= \sum_{n,t,i,j \ge 0} c(n,t,i,j) x^n y^t z^i w^j.
\end{align*}

We have 

\begin{thm}\label{teo2}
    The generating function of $c(n,t,i,j)$ is 
    $$ \Psi(x,y,z,w) =  \frac{1 - x(1 + x^m z w^{m-1} (1 + z)}{(1-x) (1 - (x(1 + x^m z w^{m-1}(1 + z) ) + y(1 + xz) ) )}.$$    . 
\end{thm}
\begin{proof} By Lemma \ref{lem_rec_2}, for all $t \ge 0$ we have 
\begin{align*}
    c(n,t&+1,i,j)   = c(n,t,i,j) + c(n-1,t+1,i,j) + c(n-1,t,i-1,j) \\
    & + c(n-m-1,t+1,i-1,j-m+1) + c(n-m-1,t+1,i-2,j-m+1).
\end{align*}
Clearly, 
$$c(n,0,i,j) = \begin{cases} 1 & \text{ if } i=j = 0 \text{ and } n \ge 0,\\
0 & \text{ otherwise}.\end{cases}$$
Let $\Psi_0 = \sum_{n,i,j \ge 0} c(n,0,i,j)x^n z^i w^j$. Then we have 

$$\Psi_0 = \sum_{n \ge 0} x^n = \frac{1}{1-x}.$$

Thus, we have 

\begin{align*}
    \Psi &= \frac{1}{1-x} + \sum_{n,i,j \ge 0, t \ge 1} c(n,t,i,j) x^n y^t z^i w^j \\
    & = \frac{1}{1-x} + \sum_{n,i,j, t\ge 1} \left [ c(n,t-1,i,j) + c(n-1,t,i,j) + c(n-1,t-1,i-1,j) \right .\\
    & + \left . c(n-m-1,t,i-1,j-m+1) + c(n-m-1,t,i-2,j-m+1) \right ] x^ny^tz^iw^j \\
    & =  \frac{1}{1-x} +  y \Psi + x (\Psi - \frac{1}{1-x}) + xyz \Psi + (x^{m+1} z w^{m-1}  + x^{m+1} z^2 w^{m-1} ) (\Psi - \frac{1}{1-x}).
\end{align*}

Hence, 

$$(1 - x(1 + x^m z w^{m-1} (1 + z))  - y(1 + xz) ) \Psi = \frac{1 - x(1 + x^m z w^{m-1} (1 + z))}{1-x}. $$

The conclusion follows.
\end{proof}
We are now ready for the proof of Theorem \ref{teo1}.
\begin{proof}[Proof of Theorem \ref{teo1}]
    Since $I_{n,m}^t$ is generated in degree $tm$, we have $\beta_{i,j}(I_{n,m}^t) = 0 \text{ if  } j < i + tm.$ By definition, $b(n,t,i,i+tm+j) = c(n,t,i,j)$ is the coefficient of $x^ny^tz^iw^j$ of $\Psi(x,y,z,w)$. Furthermore, since $w$ only appears in $\Psi$ with exponent $m-1$, we deduce that $c(n,t,i,j) = 0$ unless $j = (m-1)\ell$ for some $\ell \ge 0$. 

    Now assume that $j = (m-1) \ell$ for some integer $\ell \ge 0$. We prove by induction on $t$ the formula for $c(n,t,i,(m-1)\ell)$. The base case $t = 0$ is clear. Let 
    $$d(n,t,i,\ell) = \binom{t+\ell-1}{\ell}\binom{n-\ell m}{i-\ell}\binom{n+t-\ell m-i+\ell}{t-i+2\ell}.$$
    Since $c(n,0,t,i,(m-1)\ell) = d(n,0,i,\ell)$ for all $n,i,\ell$, for the induction step, we need to prove the following equation for all $t \ge 0$.
    \begin{align*}
    d(n,t&+1,i, \ell )   = d(n,t,i, \ell ) + d(n-1,t+1,i, \ell ) + d(n-1,t,i-1, \ell ) \\
    & + d(n-m-1,t+1,i-1,  \ell-1 ) + d(n-m-1,t+1,i-2,\ell-1).
\end{align*}
We have 
\begin{align*}
d(n,t+1,i, \ell )  &= \binom{t+\ell}{\ell}\binom{n-\ell m}{i-\ell}\binom{n+t-\ell m-i+\ell+1}{t-i+2\ell+1}\\
    d(n-1,t+1,i, \ell ) &= \binom{t+\ell}{\ell}\binom{n-\ell m-1}{i-\ell}\binom{n+t-\ell m-i+\ell}{t-i+2\ell+1}\\
    d(n-1,t,i-1, \ell ) &=\binom{t+\ell-1}{\ell}\binom{n-\ell m-1}{i-\ell-1}\binom{n+t-\ell m-i+\ell}{t-i+2\ell+1}\\
    d(n-m-1,t+1,i-1,\ell-1) &= \binom{t+\ell-1}{\ell-1}\binom{n-\ell m-1}{i-\ell}\binom{n+t-\ell m-i+\ell}{t-i+2\ell}\\
    d(n-m-1,t+1,i-2,\ell-1) &= \binom{t+\ell-1}{\ell-1}\binom{n-\ell m-1}{i-\ell -1}\binom{n+t-\ell m-i+\ell+1}{t-i+2\ell+1}
\end{align*}
We have the following binomial identities
\begin{align*}
    \binom{t+\ell}{\ell} &= \binom{t+\ell-1}{\ell} + \binom{t+\ell-1}{\ell-1}\\
    \binom{n-\ell m}{i-\ell} & = \binom{n-\ell m-1}{i-\ell} + \binom{n-\ell m-1}{i-\ell-1}\\
    \binom{n+t-\ell m-i+\ell+1}{t-i+2\ell+1}&= \binom{n+t-\ell m-i+\ell}{t-i+2\ell+1} + \binom{n+t-\ell m-i+\ell}{t-i+2\ell}.
\end{align*}
The conclusion follows from a routine check. 
\end{proof}

We now have some applications of our main result.

\begin{cor}\label{cor_t1} The Betti number $\beta_{i,j}(I_{n,m})$ is equal to $0$ unless $j = i + m + (m-1)\ell$ for some $\ell \ge 0$ and 
$$\beta_{i,i + m + (m-1)\ell}(I_{n,m}) = \binom{n - \ell m}{i - \ell} \binom{n + 1 - \ell m - i + \ell}{2 \ell + 1 -i}.$$    
\end{cor}
\begin{proof}
    Applying Theorem \ref{teo1} with $t = 1$, the conclusion follows.
\end{proof}

\begin{cor}\label{cor_l0} For all integers $i \ge 0$ and $t\ge 1$, we have 
$$\beta_{i,i + tm}(I_{n,m}^t) = \binom{n}{i} \binom{n + t - i}{t -i}.$$    
\end{cor}
\begin{proof}
    Applying Theorem \ref{teo1} with $\ell = 0$, the conclusion follows.
\end{proof}

\begin{cor}\label{cor_reg} Assume that $n \ge 0$ and $t \ge 1$. Then 
$$\reg (I_{n,m}^t) = t m  + (m-1) \left  \lfloor \frac{n }{m+1} 
 \right \rfloor.$$  
\end{cor}
\begin{proof} By Theorem \ref{teo1}, the Betti numbers $\beta_{i,i+tm + (m-1)\ell} (I_{n,m}^t)$ is non-zero if and only if the following system of inequalities have non-negative integer solutions 
\begin{equation}\label{eq_non_zero}
\begin{split}
n - \ell m &\ge \max \{ \ell , i - \ell\}, \\
    \min \{i - \ell, 2 \ell +t - i\} &\ge 0.
\end{split}   
\end{equation}
In particular, $\ell \le \left \lfloor \frac{n }{m+1} \right \rfloor$. Furthermore, when $\ell = \left \lfloor \frac{n }{m+1} \right \rfloor$, the system Eq.\eqref{eq_non_zero} has a non-negative integer solution $i = \ell$. Hence, 
\begin{align*}
    \reg (I_{n,m}^t) &= \max \{j - i \mid \beta_{i,j}(I_{n,m}^t) \neq 0 \} \\
    &=tm + (m-1) \left \lfloor \frac{n }{m+1} \right \rfloor.
\end{align*}
The conclusion follows.
\end{proof}

\begin{cor}\label{cor_pd} Assume that $n \ge 0$ and $t \ge 1$. Then 
$$\pd (I_{n,m}^t) = \min \left \{n,  t - 1 +  \left \lfloor \frac{n-t+1}{m+1} \right \rfloor + \left \lceil \frac{n-t+1}{m+1} \right \rceil \right \}.$$    
\end{cor}
\begin{proof}
    By definition $p(n,t):=\pd(I_{n,m}^t) = \max \{ i \mid \beta_{i,j}(I_{n,m}^t) \neq 0\}$. By Eq. \eqref{eq_non_zero}, we deduce that 
    \begin{equation}\label{eq_pd}
        p(n,t) = \max \left \{ \min \{ t + 2\ell, n - (m-1) \ell\} \mid \ell = 0, \ldots  \left \lfloor \frac{n }{m+1} \right \rfloor \right \}.
    \end{equation}
    If $t \ge n$, then clearly, $p(n,t) = n$. Now, assume that $t \le n$. Let $q = \left \lfloor \frac{n-t }{m+1} \right \rfloor$. Then Eq. \eqref{eq_pd} becomes 
    \begin{align*}
        p(n,t) &=  \max \left \{ \max \{ t + 2\ell \mid \ell \le q\} , \max \left \{n - (m-1) \ell \mid q+1 \le \ell \le \left \lfloor \frac{n }{m+1} \right \rfloor  \right \}  \right \} \\
        &= \max\{t + 2q, n - (m-1)(q+1)\}\\
        &=t - 1 +  \left \lfloor \frac{n-t+1}{m+1} \right \rfloor + \left \lceil \frac{n-t+1}{m+1} \right \rceil.
    \end{align*}
    The conclusion follows.
\end{proof}
\begin{obs}
    \begin{enumerate}
        \item Corollary \ref{cor_t1} gives a simpler formula for \cite[Theorem 4.14]{AF}.
        \item Corollary \ref{cor_l0} gives the Betti numbers on the linear strand of $I_{n,m}^t$ which is simpler than \cite[Corollary 2.4]{SL}.
        \item Corollary \ref{cor_reg} is \cite[Theorem 3.6]{SL}.
        \item Corollary \ref{cor_pd} is \cite[Corollary 2.8]{BC1}.
        \item It is not straightforward to derive the formula for the multiplicity of $I_{n,m}^t$ given in \cite{SWL} from Theorem \ref{teo1}.
        \item The homological invariants of path ideals of cycles are much more complicated. See \cite{BC2, BCV, MTV} for some partial results.
    \end{enumerate}
\end{obs}


\vspace{1mm}
\subsection*{Data availability}

Data sharing does not apply to this article as no data sets were generated or analyzed during the current study.

\subsection*{Conflict of interest}

The authors have no relevant financial interests to disclose.

\end{document}